\documentclass[a4paper,12pt]{amsart}
\usepackage{amsmath}
\usepackage{amsfonts}
\usepackage{amssymb}
\usepackage{amsthm}

\usepackage{hyperref}
\usepackage{enumitem}
\usepackage{subcaption}
\usepackage{graphicx}

\usepackage[margin=3.5cm]{geometry}

\usepackage{tikz}
\tikzset{>=latex}

\theoremstyle{plain}
\newtheorem{theorem}{Theorem}
\newtheorem{lemma}[theorem]{Lemma}
\newtheorem{proposition}[theorem]{Proposition}

\newcommand{\defn}{\emph}

\title{Small sums of five roots of unity}
\author{Ben Barber}
\address{University of Manchester and Heilbronn Institute for Mathematical Research, UK.}

\begin{document}

\begin{abstract}
Motivated by questions in number theory, Myerson asked how small the sum of 5 complex $n$th roots of unity can be.
We obtain a uniform bound of $O(n^{-4/3})$ by perturbing the vertices of a regular pentagon, improving to $O(n^{-7/3})$ infinitely often.

The corresponding configurations were suggested by examining exact minimum values computed for $n \leq 221000$.
These minima can be explained at least in part by selection of the best example from multiple families of competing configurations related to close rational approximations.
\end{abstract}

\maketitle

\section{Introduction}

Motivated by questions in number theory, Myerson~\cite{myerson} asked the following: 
what is the smallest non-zero absolute value of a sum of $k$ complex $n$th roots of unity?  
Call this minimum $f(k,n)$.
Myerson's best bounds for large $k$ and $n$ have the shape
\begin{equation} \label{general-bounds}
k^{-n} \leq f(k,n) \leq n^{-k/4 + o(1)},
\end{equation}
with the upper bound valid only when both $k$ and $n$ are even.
Tao~\cite{MO} asked the same question on MathOverflow, indicating similar bounds.

The upper bound begins with a maximum-sized set $S$ of $(n/2)$th roots linearly independent over $\mathbb Q$.
The sums of $k/2$ elements of $S$ are distinct and contained within a disc of radius $k/2$, so by the pigeonhole principle two of the sums must be close;
this small difference is a small sum of $k$ roots since $-1$ is an $n$th root when $n$ is even.
The lower bound uses the observation that a sum of $k$ roots of unity is an algebraic integer, so the product of its at most $n$ conjugates, each of absolute value at most $k$, is an integer.
For $n$ prime and distinct roots Konyagin and Lev~\cite{konyagin-lev} added a Fourier perspective to improve the lower bound to $k^{-n/4}$.
No improvements have been made to either the general lower or upper bounds since.

The cases $k \leq 4$ can be treated exactly because they have limited degrees of freedom.
Myerson describes geometric arguments in~\cite{myerson}.
We give a more detailed presentation in Section~\ref{sec:k-leq-4}, both to illustrate some ideas that we will use later and to correct one of Myerson's values for $f(3,n)$.

This leaves $k=5$ of particular interest, as it is susceptible to neither exact analysis nor naive pigeonhole arguments.
Myerson wrote (using $N$ for what we have called $n$) that `Choosing roots near the vertices of the regular $k$-gon will never result in a non-zero sum smaller in magnitude than $c_kN^{-1}$ for some constant $c_k$; nevertheless, we know of no general construction better than this.'
The main result of this paper is that this is too pessimistic: careful perturbation of the fifth roots of unity can always produce a sum of size $O(n^{-4/3})$.

\begin{theorem}\label{thm:four-thirds}
There is an absolute constant $C > 0$ such that,
for every $n \in \mathbb N$, $f(5,n) \leq Cn^{-4/3}$.
\end{theorem}

Moreover, for infinitely many $n$ the same construction achieves much more.

\begin{theorem}\label{thm:beat2}
There is an absolute constant $C > 0$ such that,
for infinitely many $n \in \mathbb N$, $f(5,n) \leq Cn^{-7/3}$.
\end{theorem}

We prove Theorems~\ref{thm:four-thirds} and~\ref{thm:beat2} in Sections~\ref{sec:speed-run},~\ref{sec:pentagons} and~\ref{sec:best}.

For small even $k$, Myerson improved the upper bound in \eqref{general-bounds} to $n^{-k/2}$ using known cases of the Prouhet--Tarry--Escott problem, which asks for non-degenerate integer solutions to
\[
a_1^j + \cdots + a_m^j = b_1^j + \cdots + b_m^j \qquad \text{for } 0 \leq j \leq m-1.
\]
These lead to short sums for even $n$ by looking at the Taylor expansion of
\begin{equation}\label{pte-construction}
e(a_1/n) + \cdots + e(a_m/n) - e(b_1/n) - \cdots - e(b_m/n), 
\end{equation}
where we write $e(x) = \exp(2\pi i x)$ and take $k=2m$.
The known solutions, for $1 \leq m \leq 10$ and $m = 12$, are listed at~\cite{pte}.

Both \eqref{pte-construction} and Theorem~\ref{thm:four-thirds} obtain short sums by perturbing a set of roots which sum to exactly $0$.
In Section~\ref{sec:two-and-three} we discuss what can be achieved by perturbing other sets of roots which sum to zero.
We make no concrete improvements to Theorems~\ref{thm:four-thirds} or~\ref{thm:beat2} but present some alternative approaches which are more likely to lead to further progress.

\begin{figure}
\input{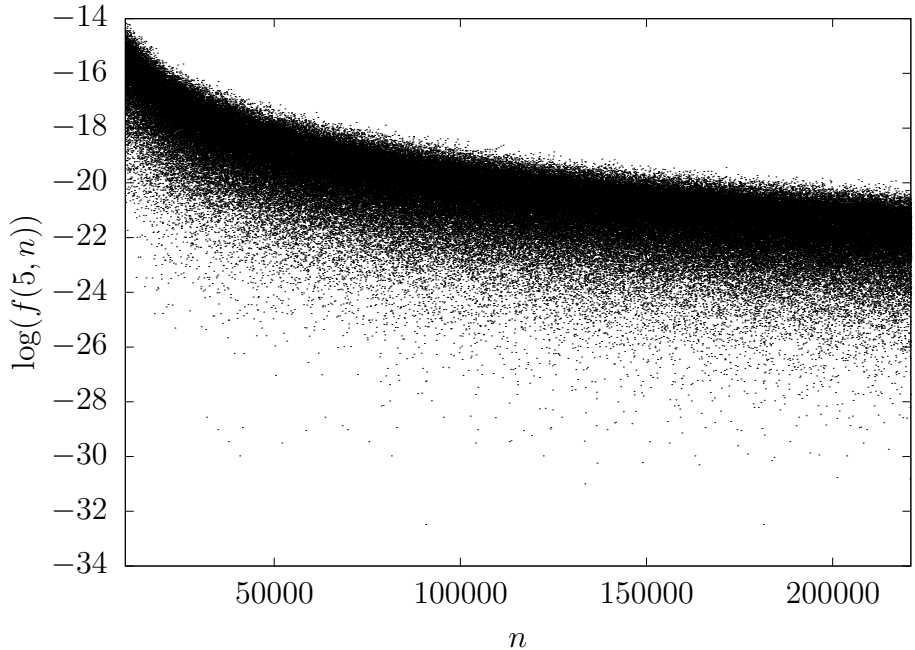}
\caption{$\log(f(5,n))$ for $10000 \leq n \leq 221000$}
\label{fig:all}
\end{figure}

\begin{figure}
\input{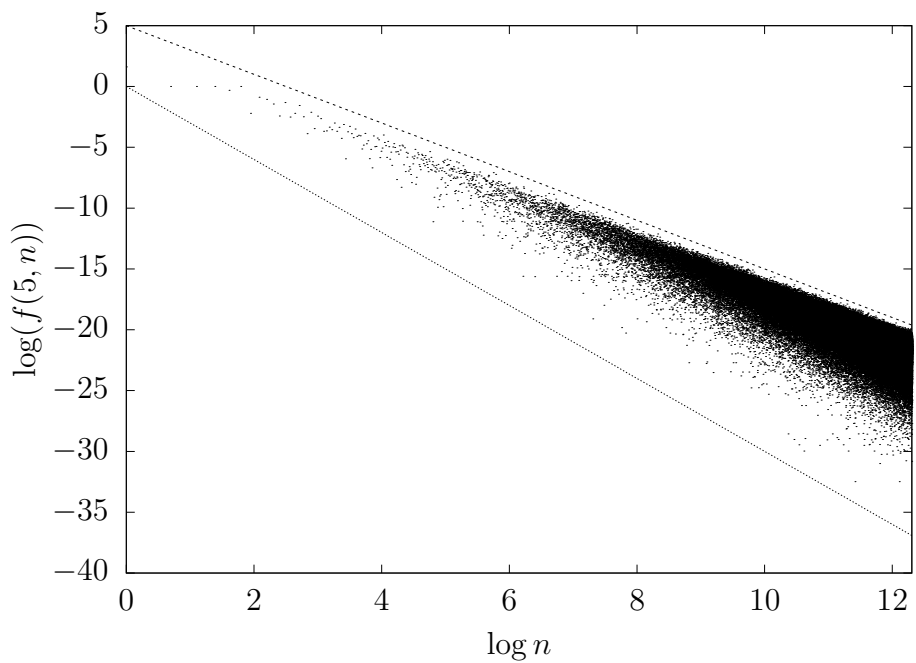}
\caption{log-log plot of $f(5,n)$ for $1 \leq n \leq 221000$ with lines of slope $-2$ and $-3$}
\label{fig:log}
\end{figure}

The starting point for this work was computation of large numbers of exact values of $f(5,n)$.
These have now been calculated for all $n \leq 221000$; see Figures~\ref{fig:all} and~\ref{fig:log} for an overview, and Section~\ref{sec:computation} for comments on how the values were obtained.
It would be bold to conjecture that either of $C/n^2$ or $c/n^3$ are upper or lower bounds respectively based on these data.

\begin{figure}
\subcaptionbox{$\log(f(5,n))$ for $10000 \leq n \leq 221000$, $n$ divisible by $6$}[.475\textwidth]{\scalebox{0.5}{\input{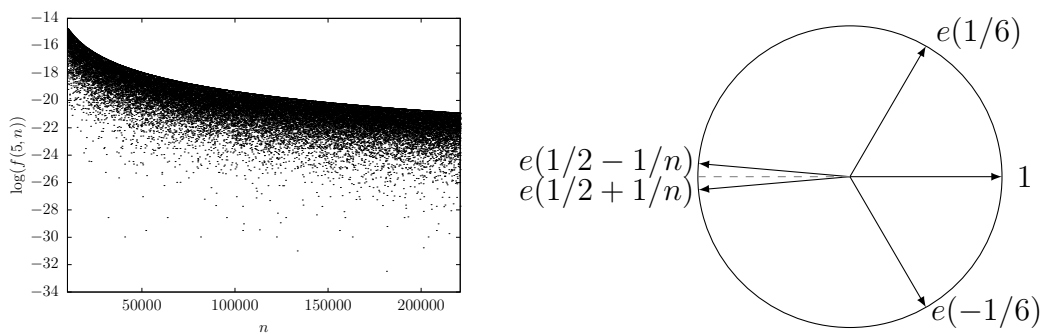}}}
\subcaptionbox{Lift of the optimal configuration for $k=4$.}[.475\textwidth]{
\begin{tikzpicture}[scale=2]
\draw (0,0) circle (1);

\foreach \t in {60,-60, 0, 175, -175} 
  \draw[->] (0,0) -- (\t:1);

\foreach \t in {180}   
  \draw[dashed,gray] (0,0) -- (\t:1);

\draw (0.85,0.95) node {$e(1/6)$};
\draw (0.9,-0.9) node {$e(-1/6)$};
\draw (-1.6,-0.1) node {$e(1/2+1/n)$};
\draw (-1.6,0.1) node {$e(1/2-1/n)$};
\draw (1.15,0) node {$1$};
\end{tikzpicture}
}
\caption{An upper bound when $6$ divides $n$ and the configuration responsible.}
\label{fig:0mod6}
\end{figure}

There is visibly a lot of variation in $f(5,n)$, with values differing by as much as a factor $e^{14} > 10^6$ for nearby $n$.
We will show that these plots conceal a great deal of structure.
For example, a distinct line is visible near the top of the cloud of points in Figure~\ref{fig:all}.
This becomes very clear if we restrict to those $n$ divisible by $6$ (Figure~\ref{fig:0mod6}).
In this case, configurations of $k$ roots can be lifted to configurations of $k+1$ roots using the fact that  $1 = e(1/6) + e(-1/6)$, so the top edge of this plot is the lift of the optimal configuration for $k=4$ (see Proposition~\ref{prop:small-k}(c)) to $k=5$.
It is possible to do better than this configuration, and $f(5,n)$ is typically smaller than $f(4,n)$, but since it is always available it provides an upper bound.
Throughout this paper we will see other families of configurations providing upper bounds in certain ranges or congruence classes of $n$.
The best picture we have of $f(5,n)$ as a whole is that of random variation within an envelope defined by these families of competing local constructions.

I would like to thank Katherine Staden for introducing me to this problem and directing me to the discussion on MathOverflow~\cite{MO}, and Jonathan Bober, Thomas Bloom, Elena Yudovina, Tom\'as Oliveira e Silva and the anonymous referee for valuable comments.

\section{Exact treatment for $k \leq 4$}\label{sec:k-leq-4}

In addition to $e(x) = \exp(2\pi i x)$, we use the less standard notation $c(x) = \cos (2\pi x)$ and $s(x) = \sin (2 \pi x)$ where this serves to clarify the presentation.

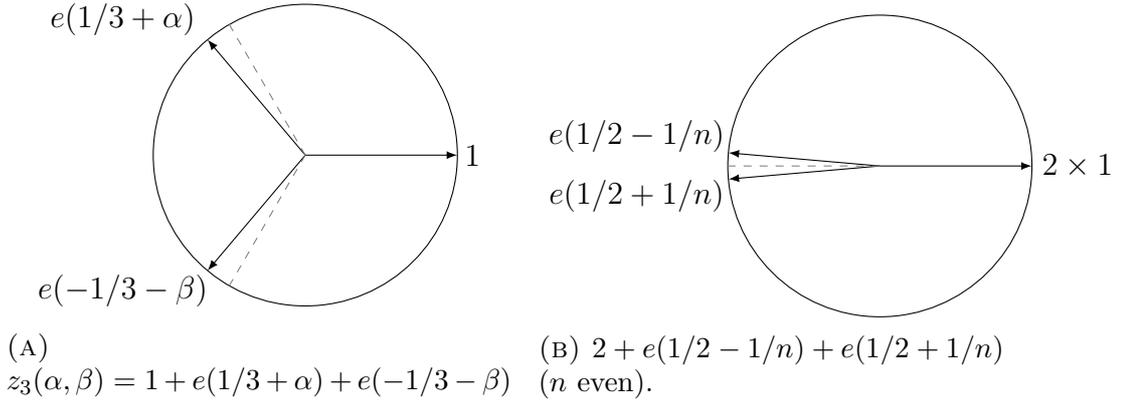
\begin{figure}
\centering
\hfill
\subcaptionbox{$z_3(\alpha,\beta) = 1+e(1/3+\alpha)+e(-1/3-\beta)$\label{fig:k=3}}[.475\textwidth]{
\centering
\begin{tikzpicture}[scale=2]
\draw (0,0) circle (1);

\foreach \t in {0,130,-130} 
  \draw[->] (0,0) -- (\t:1);

\foreach \t in {120,-120}   
  \draw[dashed,gray] (0,0) -- (\t:1);

\draw (0:1.1) node {$1$};
\draw (-1.2,0.9) node {$e(1/3+\alpha)$};
\draw (-1.2,-0.9) node {$e(-1/3-\beta)$};
\end{tikzpicture}
%
}
\hfill
\subcaptionbox{$2+e(1/2-1/n)+e(1/2+1/n)$\\ ($n$ even).\label{fig:k=4}}[.475\textwidth]{
\centering
\begin{tikzpicture}[scale=2]
\draw (0,0) circle (1);

\foreach \t in {0,175,-175} 
  \draw[->] (0,0) -- (\t:1);

\foreach \t in {180}   
  \draw[dashed,gray] (0,0) -- (\t:1);

\draw (0:1.3) node {$2\times1$};
\draw (-1.6,0.2) node {$e(1/2-1/n)$};
\draw (-1.6,-0.2) node {$e(1/2+1/n)$};

\end{tikzpicture}
%
}
\hfill
\caption{Configurations for $k=3$ and $k=4$.}

\end{figure}

\begin{proposition}\label{prop:small-k}
\begin{itemize}
\item[(a)] For all $n$,
\[
f(2,n) =
\begin{cases}
2\sin(\pi/n) & \text{if $n$ is even}, \\
2\sin(\pi/2n) & \text{if $n$ is odd.}
\end{cases}
\]

\item[(b)] For $n$ sufficiently large,
\[
f(3,n) = 
\begin{cases}
2\sin(\pi/3n) & \text{if }3|n, \\
\sqrt 3 \sin(2\pi/3n) - 2\sin^2(\pi/3n) & \text{if }n \equiv -1 \mod 3, \\
\sqrt 3 \sin(2\pi/3n) + 2\sin^2(\pi/3n) & \text{if }n \equiv 1 \mod 3.
\end{cases}
\]

\item[(c)] For $n$ sufficiently large,
\[
f(4,n) =
\begin{cases}
4\sin^2(\pi/n) & \text{if $n$ is even}, \\
4\sin^2(\pi/2n) & \text{if $n$ is odd}.
\end{cases}
\]
\end{itemize}
\end{proposition}

\cite{myerson} gives the value $2\pi\sqrt 3/n + O(1/n^2)$ for $f(3,n)$ when $3|n$, which can be seen to be incorrect by considering $1 + e(1/3) + e(2/3 + 1/n)$, of length $2\pi/n + O(1/n^2)$.

The optimal configurations for parts (a) and (c) of Proposition~\ref{prop:small-k} are related, so we treat those cases first.

\begin{proof}
(a)
A sum of two $n$th roots of unity has the form
\begin{equation}\label{eqn:sum-of-2}
e(a/n) + e(b/n) = 2c((a-b)/2n))e((a+b)/2n).
\end{equation}
The absolute value of the cosine takes its minimum value non-zero when $b = 0$ and $a$ is the greatest integer less than $n/2$.
Then $e(a/n)$ and $e(0/n)$ are two almost opposite roots whose sum has length the claimed value of $f(2,n)$.

(c) With the same value of $a$, $e(0/n) + 2e(a/n) + e(2a/n)$ is a sum of four roots of length $f(2,n)^2$, establishing the upper bound for $f(4,n)$.

For the lower bound, let $z = u + v + w + x$ be a smallest non-zero sum of four $n$th roots of unity and let $n$ be sufficiently large.
Assume without loss of generality that the angle between $u$ and $v$ is at most $\pi/2$, so that $|u+v| \geq \sqrt 2$.

If $u+v$, $0$ and $w+x$ were not collinear, then the angle they formed at $0$ would be at least $\pi/n$, implying that $|z| \geq \sqrt 2 \sin (\pi/n)$, larger than the upper bound.
Thus $u+v$ and $w+x$ are on opposite sides of the same line through $0$.

Since $|u+v|\geq \sqrt 2$, we must have $|w+x| \geq \sqrt 2 - O(1/n^2)$, and so the angle between $w$ and $x$ is at most $\pi/2 + O(1/n)$.
Hence we may write
\begin{align*}
u+v & = 2c((a-b)/2n)e((a+b)/2n) \\
w+x & = 2c((p-q)/2n)e((p+q)/2n),
\end{align*}
where $0 \leq a - b \leq n/4$ and $0 \leq p-q \leq n/4 + O(1)$, and so both cosines are positive.

To ensure that $u+v$, $w+x$ are pointing in opposite directions, we require that
\[
a + b \equiv p + q + n \mod {2n},
\]
whence
\[
a + b \equiv p + q + n \mod {2},
\]
and finally
\[
a - b \equiv p - q + n \mod {2}.
\]
Now $|z| = 2|\cos(\pi(a-b)/n) - \cos(\pi(p-q)/n)|$ takes its least non-zero value when $a-b=0$, $p-q=1$.
This is consistent with the parity condition when $n$ is odd.
When $n$ is even we must instead take $a-b=0$, $p-q=2$.
Recalling that $1-\cos 2\theta  = 2\sin^2 \theta$, these values give the claimed lower bound for $f(4,n)$.

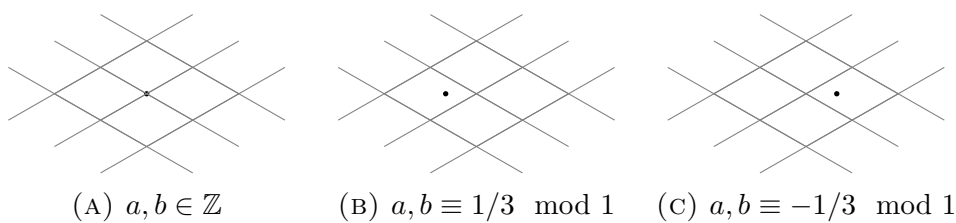
\begin{figure}
\centering
\subcaptionbox{$a, b \in \mathbb Z$\label{fig:tri}}[0.3\linewidth]{
\begin{tikzpicture}[scale=0.7]
\fill (0,0) circle (0.05);
\foreach \a in {-1, ..., 1}
{
\foreach \b in {-1, ..., 1}
{
  \draw[gray] ({(\a+\b)*cos(30)},{(\a-\b)*sin(30)}) -- ({(\a+\b+1)*cos(30)},{(\a-\b+1)*sin(30)});
  \draw[gray] ({(\a+\b)*cos(30)},{(\a-\b)*sin(30)}) -- ({(\a+\b+1)*cos(30)},{(\a-\b-1)*sin(30)});
  \draw[gray] ({(\a+\b)*cos(30)},{(\a-\b)*sin(30)}) -- ({(\a+\b-1)*cos(30)},{(\a-\b+1)*sin(30)});
  \draw[gray] ({(\a+\b)*cos(30)},{(\a-\b)*sin(30)}) -- ({(\a+\b-1)*cos(30)},{(\a-\b-1)*sin(30)});
}
}
\end{tikzpicture}
}
\subcaptionbox{$a, b \equiv 1/3 \mod 1$\label{fig:tri+1/3}}[0.3\linewidth]{
\begin{tikzpicture}[scale=0.7]
\fill (0,0) circle (0.05);
\foreach \a in {-1, ..., 1}
{
\foreach \b in {-1, ..., 1}
{
  \draw[gray] ({(\a+\b+0.6666)*cos(30)},{(\a-\b)*sin(30)}) -- ({(\a+\b+0.6666+1)*cos(30)},{(\a-\b+1)*sin(30)});
  \draw[gray] ({(\a+\b+0.6666)*cos(30)},{(\a-\b)*sin(30)}) -- ({(\a+\b+0.6666+1)*cos(30)},{(\a-\b-1)*sin(30)});
  \draw[gray] ({(\a+\b+0.6666)*cos(30)},{(\a-\b)*sin(30)}) -- ({(\a+\b+0.6666-1)*cos(30)},{(\a-\b+1)*sin(30)});
  \draw[gray] ({(\a+\b+0.6666)*cos(30)},{(\a-\b)*sin(30)}) -- ({(\a+\b+0.6666-1)*cos(30)},{(\a-\b-1)*sin(30)});
}
}
\end{tikzpicture}
}
\subcaptionbox{$a, b \equiv -1/3 \mod 1$\label{fig:tri-1/3}}[0.3\linewidth]{
\begin{tikzpicture}[scale=0.7]
\fill (0,0) circle (0.05);
\foreach \a in {-1, ..., 1}
{
\foreach \b in {-1, ..., 1}
{
  \draw[gray] ({(\a+\b-0.6666)*cos(30)},{(\a-\b)*sin(30)}) -- ({(\a+\b-0.6666+1)*cos(30)},{(\a-\b+1)*sin(30)});
  \draw[gray] ({(\a+\b-0.6666)*cos(30)},{(\a-\b)*sin(30)}) -- ({(\a+\b-0.6666+1)*cos(30)},{(\a-\b-1)*sin(30)});
  \draw[gray] ({(\a+\b-0.6666)*cos(30)},{(\a-\b)*sin(30)}) -- ({(\a+\b-0.6666-1)*cos(30)},{(\a-\b+1)*sin(30)});
  \draw[gray] ({(\a+\b-0.6666)*cos(30)},{(\a-\b)*sin(30)}) -- ({(\a+\b-0.6666-1)*cos(30)},{(\a-\b-1)*sin(30)});
}
}
\end{tikzpicture}
}
\caption{$t(a/n,b/n)$, origin marked.}
\label{fig:triangular-lattices}
\end{figure}

(b) Let $z=u+v+w$ be a smallest non-zero sum of three $n$th roots of unity and let $n$ be sufficiently large.
A perturbation of the third roots of unity shows that $f(3,n) = O(1/n)$, so $|u+v| = 1 + O(1/n)$, whence the angle between $u$ and $v$ is $2\pi/3 + O(1/n)$.
Similarly, the angles between $v$ and $w$ and between $u$ and $w$ are $2\pi/3 + O(1/n)$, so we may assume that
\[
z = z_3(\alpha, \beta) = 1 + e(1/3 + \alpha) + e(-1/3 -\beta),
\]
with $\alpha, \beta = O(1/n)$.
Expanding to first order, 
\begin{align*}
z_3(\alpha, \beta) & = 2\pi\alpha ie(1/3) - 2\pi\beta ie(-1/3) + O(1/n^2)
\\ & = -2\pi(\alpha e(1/12) + \beta e(-1/12)) + O(1/n^2).
\end{align*}
Write $t(\alpha, \beta) = \alpha e(1/12) + \beta e(-1/12)$.
When $3|n$, so that $e(1/3)$ and $e(-1/3)$ are $n$th roots of unity, the legal values of $\alpha, \beta$ are integer multiples of $1/n$.
As $a, b$ vary, $t(a/n, b/n)$ describes the vertices of a triangular lattice (Figure~\ref{fig:tri}).
The smallest non-zero points in this lattice have length $1/n$, achieved by $(a, b) \in \{(\pm 1, 0), (0, \pm 1), (\pm 1, \mp 1)\}$.
Up to rotation, all six of these choices correspond to the configuration
\begin{align*}
z & = e(1/n) + e(1/3) + e(-1/3) = e(1/n) - 1
\\ & = e(1/2n)(e(1/2n)-e(-1/2n))
\\ & = 2ie(1/2n)\sin(\pi/n),
\end{align*}
which has the claimed size.

When $n \equiv -1$ mod $3$, the legal values of $\alpha$ are $a/n+1/3n$ for integer $a$, so that $e(1/3 + a/n+1/3n) = e((n + 3a + 1)/3n)$ is an $n$th root of unity.
Similarly, the legal values of $\beta$ are $b/n+1/3n$ for integer $b$.
As $a$ and $b$ vary, $t(a/n + 1/3n, b/n + 1/3n)$ describes the vertices of an offset triangular lattice, with the origin in the centre of a right-pointing triangle of side length $1/n$ (Figure~\ref{fig:tri+1/3}).
The smallest non-zero points in this offset lattice have length $1/\sqrt 3 n$, achieved by $(a, b) \in \{(0,0), (0,-1), (-1,0)\}$.
Up to rotation, all three of these choices correspond to the configuration
\begin{align*}
1 + e(1/3 + 1/3n) + e(-1/3 - 1/3n)
   & = 1 + 2c(1/3 + 1/3n)
\\ & = 1 + 2c(1/3)c(1/3n) - 2s(1/3)s(1/3n)
\\ & = 1 - \cos(2\pi/3n) - \sqrt 3 \sin(2\pi/3n)
\\ & = 2\sin^2(\pi/3n) - \sqrt 3 \sin(2\pi/3n).
\end{align*}
Similarly, when $n \equiv 1$ mod $3$, the legal values of $\alpha, \beta$ are $a/n-1/n, b/n-1/n$ for integer $a, b$, and as $a$ and $b$ vary, $t(a/n - 1/3n, b/n - 1/3n)$ describes the vertices of an offset triangular lattice, with the origin in the centre of a left-pointing triangle of side length $1/n$ (Figure~\ref{fig:tri-1/3}).
The smallest non-zero points in this offset lattice have length $1/\sqrt 3 n$, achieved by $(a, b) \in \{(0,0), (0,1), (1,0)\}$.
Up to rotation, all three of these choices correspond to the configuration
\begin{align*}
1 + e(1/3 + 1/3n) + e(-1/3 - 1/3n)
& = 2\sin^2(\pi/3n) + \sqrt 3 \sin(2\pi/3n).\qedhere
\end{align*}
\end{proof}

With a little more work and checking a handful of cases numerically one can show that the stated values of $f(4,n)$ are correct except for $f(4,2) = 2$ and $f(4,4) = \sqrt 2$.

\section{Perturbing a regular pentagon}\label{sec:speed-run}

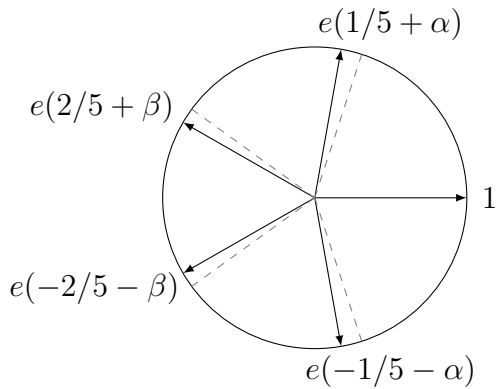
\begin{figure}
\centering

\begin{tikzpicture}[scale=2]
\draw (0,0) circle (1);

\foreach \t in {0,80,-80,150,-150} 
  \draw[->] (0,0) -- (\t:1);

\foreach \t in {72,-72,144,-144}   
  \draw[dashed,gray] (0,0) -- (\t:1);

\draw (0:1.15) node {$1$};
\draw (0.5,1.15) node {$e(1/5+\alpha)$};
\draw (0.5,-1.15) node {$e(-1/5-\alpha)$};
\draw (-1.4,0.6) node {$e(2/5+\beta)$};
\draw (-1.45,-0.6) node {$e(-2/5-\beta)$};

\end{tikzpicture}

\caption{$z_5(\alpha,\beta)$}
\label{fig:z5}
\end{figure}

We first give the simplest version of our argument demonstrating that perturbations of the regular pentagon can produce short sums when $n$ is divisible by $5$.
We restrict attention to the family of perturbations $z_5(\alpha,\beta)$ illustrated in Figure~\ref{fig:z5}, which ensures that the sum is real-valued.

We use the same basic approach as \eqref{pte-construction} of taking a Taylor expansion for our configuration.
In this section we only require an expansion to first order, but we will use higher order terms later.

We use the exact values
\[
\phi = \frac 1 {2c(1/5)} = -2c(2/5),
\]
where $\phi = \frac{1+\sqrt 5}2$ is the golden ratio.
For now this is merely a convenient simplification of notation, but we make more significant use of the appearance of $\phi$ in Section~\ref{sec:pentagons}.

\begin{lemma}\label{lem:expansion}
\begin{multline*}
z_5(\alpha,\beta) = 
- 4\pi\sin(\pi/5)[\alpha\phi+\beta] 
- 2\pi^2[\alpha^2/\phi + \beta^2\phi] \\
+ \frac{8\pi^3\sin(\pi/5)}{3}[\alpha^3\phi+\beta^3] 
+ O(\alpha^4) + O(\beta^4).
\end{multline*}
\end{lemma}

\begin{proof}
\begin{align*}
z_5(\alpha,\beta) & = 1 + e(1/5 + \alpha) + e(-1/5 - \alpha) + e(2/5 + \beta) + e(-2/5-\beta)
\\ & = 1 + 2c(1/5 + \alpha) + 2c(2/5 + \beta)
\\ & = 1 + 2\big[c(1/5)c(\alpha) - s(1/5)s(\alpha)  + c(2/5)c(\beta) - s(2/5)s(\beta)\big]
\\ & = 1 - 2s(2/5)\left[\frac{s(\alpha)}{2c(1/5)}+s(\beta)\right] + 2[c(1/5)c(\alpha) + c(2/5)c(\beta)]
\\ & = 1 - 2s(1/10)[s(\alpha)\phi+s(\beta)] + [c(\alpha)/\phi - c(\beta)\phi]
\\ & = - 2\sin(\pi/5)\Big[\big((2\pi\alpha)-(2\pi\alpha)^3/6\big)\phi+\big((2\pi\beta)-(2\pi\beta)^3/6\big)\Big] 
\\ & \qquad\qquad - {\textstyle\frac 1 2} [(2\pi\alpha)^2/\phi + (2\pi\beta)^2\phi] + O(\alpha^4) + O(\beta^4).\qedhere
\end{align*}
\end{proof}

It is impossible to make the leading term vanish non-trivially since $\phi$ is irrational.
Instead we find unusually small leading terms using Dirichlet's theorem.

\begin{theorem}[Dirichlet]
For every $Q \in \mathbb N$ and $\xi \in \mathbb R$, there are integers $p, q$ with $1 \leq q \leq Q$ such that
\[
\left|q\xi - p\right| \leq 1/Q.
\]
\end{theorem}

\begin{proposition}\label{speed-run}
There is an absolute constant $C$ such that $f(5,n) \leq Cn^{-4/3}$ whenever $n$ is divisible by $5$.
\end{proposition}

\begin{proof}
By Lemma~\ref{lem:expansion},
\[
z_5(\alpha,\beta) =-4\pi\sin(\pi/5)[\alpha\phi + \beta] + O(\alpha^2) + O(\beta^2).
\]
Since $n$ is divisible by $5$, we can write $\alpha = a/n$, $\beta = b/n$ with $a, b$ integers.
Then
\begin{align*}
|z_5(\alpha,\beta)| & = \frac{4\pi\sin(\pi/5)|a\phi + b|}n + O(a^2/n^2) + O(b^2/n^2).
\end{align*}
Let $Q$ be a positive integer to be specified later and apply Dirichlet's theorem to obtain $p, q$ such that $|q\phi-p| \leq 1/Q$ and $1 \leq q \leq Q$.
Put $a=q$ and $b=-p$ to obtain
\begin{align*}
|z_5(\alpha,\beta)| & \leq \frac{4\pi\sin(\pi/5)} {Qn} + O(Q^2/n^2).
\end{align*}
The optimal choice $Q = \Theta(n^{1/3})$ gives $|z_5(\alpha,\beta)| = O(n^{-4/3})$.
\end{proof}

Since Dirichlet's theorem is proved by the pigeonhole principle, this proof can be viewed as a precision application of the pigeonhole principle in a small part of the configuration space.

It is worth considering why this argument does not apply to the case $k=3$.
To first order, perturbation of a regular $k$-gon generates a copy of $(2\pi i/n)\cdot \mathbb Z[e(1/k)]$.
For $k=3$ this is the discrete subgroup of $\mathbb C$ we parameterised by $t(a/n,b/n)$ in the proof of Proposition~\ref{prop:small-k}, but for $k=5$ it is dense, which allows us to find non-trivial points near $0$.

A similar argument explains why the extremal configurations for $k=4$ are those shown in Figure~\ref{fig:k=4} rather than perturbations of $\{1, i, -1, -i\}$.
For perturbations of the square there is a two parameter family of ways to get back to $0$ in $(2\pi i/n)\cdot \mathbb Z[i]$, but they correspond to rotating each of $\{1, -1\}$ and $\{i, -i\}$ as pairs so vanish to all orders.
For perturbations of $\{1, 1, -1, -1\}$ as in Figure~\ref{fig:k=4} there are non-trivial ways to return to $0$ to first order, leaving $f(4,n)$ equal to the neglected quadratic term.

\section{Proof of Theorem~\ref{thm:four-thirds}}\label{sec:pentagons}

Removing the divisibility condition from Proposition~\ref{speed-run} amounts to understanding close rational approximations with congruence conditions on the numerator and denominator.
Here it is convenient that the value we seek to approximate is the golden ratio $\phi$.

Let $(F_m)$ be the Fibonacci sequence defined by $F_0 = 0$, $F_1 = 1$ and $F_m = F_{m-1} + F_{m-2}$ for $m \geq 2$.
We make repeated use of the explicit formulae
\[
F_m = \frac{\phi^m - (-\phi)^{-m}}{\sqrt 5}
\]
and
\begin{align*}
F_m\phi-F_{m+1} & = \frac{[\phi^m - (-\phi)^{-m}]\phi - [\phi^{m+1} - (-\phi)^{-(m+1)}]}{\sqrt 5}
= \frac{(-1)^{m+1}}{\phi^m}.
\end{align*}

\begin{lemma}\label{fibonacci}
There is a constant $C_0 > 1$ such that, for each $r \in \{0, \ldots, 4\}$ there are sequences of integers $(a_{j,r})_{j=0}^\infty$, $(b_{j,r})_{j=0}^\infty$ such that
\begin{gather*}
a_{j,r} \equiv r \mod 5 \\
b_{j,r} \equiv 2r \mod 5 \\
\phi^{20j}/C_0 < |a_{j,r}| < C_0\phi^{20j} \\
\phi^{20j}/C_0 < |b_{j,r}| < C_0\phi^{20j} \\
1\big/C_0|a_{j,r}| < a_{j,r}\phi + b_{j,r} < C_0/|a_{j,r}|.
\end{gather*}
\end{lemma}

\begin{proof}
We will choose $(a_{j,r})_{j=0}^\infty$, $(b_{j,r})_{j=0}^\infty$ such that there are constants $A_r$, $B_r$ and $D_r$ such that
\begin{align*}
a_{j,r} & \sim A_r \phi^{20j} \\
b_{j,r} & \sim B_r \phi^{20j}
\end{align*}
as $j \to \infty$, and
\[
a_{j,r}\phi+b_{j,r} = D_r / \phi^{20j} > 0,
\]
from which the existence of a suitable $C_0$ follows.

For $r=0$, take 
\begin{gather*}
a_{j,0} = 5F_{20j+1} \\
b_{j,0} = -5F_{20j+2}
\end{gather*}
so that
\[
a_{j,0}\phi + b_{j,0} = 5/\phi^{20j+1}.
\]

For $r \neq 0$ we first observe that, since the values of $F_0, \ldots, F_6$ mod $5$ are $0,1,1,2,3,0,3$, we have by induction that
\begin{align*}
F_{5j}   \equiv 0 \mod 5, \quad
F_{5j+1} \equiv 3^j \mod 5 \quad \text{and} \quad
F_{5j+2} \equiv 3^j \mod 5.
\end{align*}
For $r=1$, take
\begin{gather*}
   a_{j,1} = F_{20j+1} + 2F_{20j+20} \equiv 1 \cdot 1 + 2 \cdot 0 \equiv 1 \mod 5
\\ b_{j,1} = -(F_{20j+2} + 2F_{20j+21}) \equiv -(1 \cdot 1 + 2 \cdot 1) \equiv 2 \mod 5
\end{gather*}
so that
\[
a_{j,1}\phi + b_{j,1} = 1/\phi^{20j+1} - 2/\phi^{20j+20} > 0.
\]

For $r=4$ we move 10 steps along $(F_m)$ and take
\begin{gather*}
   a_{j,4} = F_{20j+11} + 2F_{20j+30} \equiv 1\cdot 3^2 \equiv 4 \mod 5
\\ b_{j,4} = -(F_{20j+12} + 2F_{20j+31}) \equiv 2\cdot 3^2 \equiv 3 \mod 5
\end{gather*}
so that
\[
a_{j,4}\phi + b_{j,4} = 1/\phi^{20j+11} - 2/\phi^{20j+30} > 0.
\]

For the remaining cases we would like to take 5 steps along $(F_m)$, but that would flip the sign of our approximations.
To correct for that we move the terms that make negative contributions further down the sequence $(F_m)$ and bring the terms that make positive contributions forward to swap their relative sizes.
Thus for $r=3$ we take
\begin{gather*}
   a_{j,3} = F_{20j+26} + 2F_{20j+5} \equiv 1\cdot3 \equiv 3 \mod 5
\\ b_{j,3} = -(F_{20j+27} + 2F_{20j+6}) \equiv 2\cdot 3 \equiv 1 \mod 5
\end{gather*}
so that
\[
a_{j,3}\phi + b_{j,3} = -1/\phi^{20j+26} + 2/\phi^{20j+5} > 0.
\]
Similarly, for $r=2$ we take
\begin{gather*}
   a_{j,2} = F_{20j+36} + 2F_{20j+15} \equiv 3\cdot 3^2 \equiv 2 \mod 5
\\ b_{j,2} = -(F_{20j+37} + 2F_{20j+16}) \equiv 1 \cdot 3^2 \equiv 4 \mod 5
\end{gather*}
so that
\[
a_{j,3}\phi + b_{j,3} = -1/\phi^{20j+36} + 2/\phi^{20j+15} > 0.\qedhere
\]
\end{proof}

\begin{proof}[Proof of Theorem~\ref{thm:four-thirds}]
We use the parameterisation $z_5(\alpha, \beta)$ with $\alpha=a/5n$, $\beta = b/5n$ and require that $a \equiv -n \mod 5$ and $b \equiv -2n \mod 5$ to ensure that $e(1/5 + \alpha)$ and $e(2/5+\beta)$ are $n$th roots of unity.
By Lemma~\ref{lem:expansion},
\begin{align*}
z_5(\alpha,\beta) & = -\textstyle\frac{4\pi\sin(\pi/5)}{5n}[a\phi + b] + O(a^2/n^2) + O(b^2/n^2).
\end{align*}

Let $r \equiv -n \mod 5$, and let $C_0$, $(a_{j,r})_{j=0}^\infty$ and $(b_{j,r})_{j=0}^\infty$ be as in Lemma~\ref{fibonacci}.
Then
\[
|z_5(a_{j,r}/5n, b_{j,r}/5n)| \leq \frac{4\pi\sin(\pi/5)C_0}{|a_{j,r}|n} + O(a_{j,r}^2/n^2) + O(b_{j,r}^2/n^2).
\]
Since $(a_{j,r})_{j=0}^\infty$ and $(b_{j,r})_{j=0}^\infty$ are geometrically distributed and their ratio is bounded, we can choose $j$ so that $|a_{j,r}|, |b_{j,r}| = \Theta(n^{1/3})$.
This choice shows that $f(5,n) = O(n^{-4/3})$, as required.
\end{proof}

We should ask whether the contortions of Lemma~\ref{fibonacci} were necessary.
The shortest answer is that the direct generalisation of Dirichlet's theorem where we impose arbitrary congruence conditions on $p$ and $q$ for arbitrary $\xi$ is false~\cite{no-extend-dirichlet}.
The weaker version in which the right-hand side is replaced by $C/q^2$, with $C$ depending on the moduli of the congruences, is true for infinitely many $q$ (see~\cite{same-modulus} for the case where $p$ and $q$ are restricted with respect to the same modulus), but this is insufficient to prove a uniform $O(n^{-4/3})$ bound. 
For a discussion of the state of the art for questions of this type see~\cite{adiceam}.

\section{Proof of Theorem~\ref{thm:beat2}}
\label{sec:best}

Lower bounds on the accuracy of rational approximations mean we can't do better than Theorem~\ref{thm:four-thirds} by working harder on the linear term.
Instead we obtain local improvements to $O(n^{-7/3})$ by choosing $n$ so that the linear and higher order terms approximately cancel.

A first attempt at this argument is to fix an $a$ and $b$ so that $|a\phi+b| \leq C_0/a$.
As before, there are non-zero constants $c_1, c_2$ such that
\[
z_5(a/n,b/n) = \frac {c_1} {an} + \frac{c_2a^2}{n^2} + \Theta\!\left(\frac{a^3}{n^3}\right).
\]
Then there is a real $n=n_0 = \Theta(a^3)$ so that the linear and quadratic terms cancel, and by careful choice of $a$ and $b$ we can ensure that $n_0$ is positive.
Let $n$ be the closest multiple of $5$ to $n_0$.
Then
\begin{equation}\label{eqn:swamp}
z_5(a/n,b/n) = O\!\left(\frac{1}{an^2}\right) + O\!\left(\frac{a^2}{n^3}\right) + \Theta\!\left(\frac{a^3}{n^3}\right) = O\!\left(\frac{1}{n^{7/3}}\right) + \Theta\!\left(\frac{1}{n^2}\right)\!, 
\end{equation}
which would be the result we're aiming for if the main term were not swamped by the error term.

Working to third order turns out to be sufficient.

\begin{proof}[Proof of Theorem~\ref{thm:beat2}]
Write $\alpha = a/n$, $\beta = b/n$.
By Lemma~\ref{lem:expansion},
\begin{align*}
z_5(\alpha,\beta) & = -\frac{4\pi\sin(\pi/5)[a\phi+b]n^2 + 2\pi^2[a^2/\phi  - b^2\phi]n - \frac 8 3\pi^3\sin(\pi/5)[a^3\phi + b^3]}{n^3}
\\& \qquad \qquad \qquad\qquad\qquad\qquad\qquad\qquad\qquad\qquad+ O(a^4/n^4) + O(b^4/n^4)
\end{align*}

Write $g(n) = An^2 + Bn + C$ for the quadratic in the numerator and
let $a = F_{2j-1}$, $b = -F_{2j}$ for some $j$.
Since $0 < a\phi + b = \Theta(1/a)$, we have for sufficiently large $j$ that
\begin{align*}
                     0 & < A = \Theta(1/a)
\\ \Theta(a^2) = B < 0 &
\\                   0 & < C = \Theta(a^3).
\end{align*}
Hence $g$ has a positive root
\begin{align*}
n_0 & = -\frac{B}{2A}\left(1 + \sqrt{1 -\frac{4AC}{B^2}}\right) = \Theta(a^3) \big(1 + \sqrt{1-\Theta(1/a^2)}\big) = \Theta(a^3).
\end{align*}
Let $n = n_0 + \gamma$ be the closest multiple of $5$ to $n_0$.
Then
\[
g(n) = g(n_0 + \gamma) - g(n_0) = A(2\gamma n_0 + \gamma^2) + B\gamma = O(a^2),
\]
so
\begin{align*}
z_5(a/n,b/n) & = O(a^2/n^3) +  O(a^4/n^4) = O(n^{-7/3}). \qedhere
\end{align*}
\end{proof}

We proved Theorem~\ref{thm:beat2} for $n$ divisible by $5$ for simplicity, but the same argument works in any congruence class mod $5$ with appropriate use of Lemma~\ref{fibonacci}.

The configurations arising in the proof of the $O(n^{-7/3})$ upper bound are not artefacts: see Figure~\ref{fig:fibonacci} for the dip in $f(5,n)$ caused by $z_5(13/n, -21/n)$.

\begin{figure}
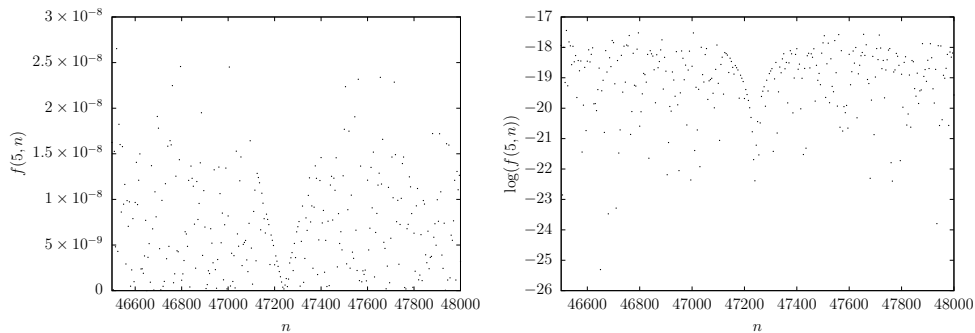

\scalebox{0.5}{\input{fibonaccilog.tex}}
\scalebox{0.5}{\input{fibonacci.tex}}
\caption{$f(5,n)$ and $\log(f(5,n))$ for $46500 \leq n \leq 48000$, $n \equiv 0 \mod 5$ showing upper bound from $z_5(13/n, -21/n)$}
\label{fig:fibonacci}
\end{figure}

\section{Perturbing other families}
\label{sec:two-and-three}

We say that a sum of $n$th roots of unity equalling $0$ is \defn{minimal} if it is not the empty sum and no non-empty proper subset of the summands sums to $0$.

\begin{theorem}[\cite{centrifuge}]
Up to rotation, the minimal sums of $n$th roots of unity equalling $0$ are either $1 + e(1/p) + \cdots + e((p-1)/p)$ for a prime $p$ dividing $n$, or have at least $(p-1)(q-1) + (r-1) \geq 6$ elements for $p < q < r$ the three smallest prime factors of $n$.
The unique sum witnessing the second bound (up to rotation) is
\[
[e(1/p) + \cdots + e((p-1)/p)][e(1/q) + \cdots + e((q-1)/q)] + e(1/r) + \cdots + e((r-1)/r).
\]
\end{theorem}

It follows that, apart from the regular pentagon, whose perturbations were analysed in Sections~\ref{sec:speed-run} and~\ref{sec:pentagons}, the only other way for five $n$th roots to sum to $0$ is to take the points of an equilateral triangle and two diametrically opposite points.
Figure~\ref{fig:two-and-three} shows (perturbations of) the two configurations of this type which are symmetric about the real axis.

\begin{figure}
\centering
\begin{tikzpicture}[scale=2]
\begin{scope}[shift={(-1.8,0)}]
\draw (0,0) circle (1);

\foreach \t in {0,130,-130, 80, -80} 
  \draw[->] (0,0) -- (\t:1);

\foreach \t in {120,-120, 90, -90}   
  \draw[dashed,gray] (0,0) -- (\t:1);

\draw (0:1.1) node {$1$};
\draw (-1.2,0.9) node {$e(1/3+\alpha)$};
\draw (-1.2,-0.9) node {$e(-1/3-\alpha)$};
\draw (0.3,1.15) node {$e(1/2+\beta)$};
\draw (0.3,-1.17) node {$e(-1/2-\beta)$};

\draw (0,-1.5) node {$z_{3,\pm i}(\alpha,\beta)$};

\end{scope}
\begin{scope}[shift={(1.8,0)}]
\draw (0,0) circle (1);

\foreach \t in {50,-50, 0, 175, -175} 
  \draw[->] (0,0) -- (\t:1);

\foreach \t in {180, 60,-60}   
  \draw[dashed,gray] (0,0) -- (\t:1);

\draw (1.15,0.9) node {$e(1/6+\alpha)$};
\draw (1.2,-0.9) node {$e(-1/6-\alpha)$};
\draw (-1.5,-0.1) node {$e(1/2+\beta)$};
\draw (-1.5,0.1) node {$e(1/2-\beta)$};
\draw (1.2,0) node {$1$};

\draw (0,-1.5) node {$z_{3,\pm 1}(\alpha,\beta)$};
\end{scope}
\end{tikzpicture}
\caption{Two more families of configurations.}
\label{fig:two-and-three}
\end{figure}
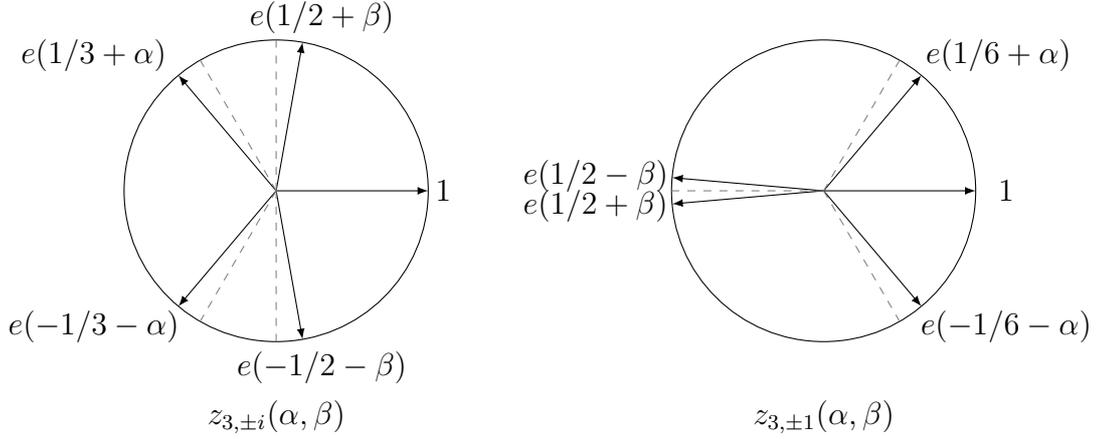

The first configuration behaves similarly to $z_5(\alpha, \beta)$, so could be used in its place to prove Theorems~\ref{thm:four-thirds} (given a suitable version of Lemma~\ref{fibonacci}) and~\ref{thm:beat2}, using close rational approximations of $\sqrt 3/2$ instead of $\phi$:
\begin{align*}
z_{3,\pm i}(\alpha,\beta) &=  1 + e(1/3 + \alpha) + e(-1/3-\alpha) + e(1/4 + \beta) + e(-1/4-\beta)
\\ & = 1 + 2c(1/3 + \alpha) + 2c(1/4 + \beta)
\\ & = 1 + 2[c(1/3)c(\alpha) - s(1/3)s(\alpha)] - 2s(\beta)
\\ & = 1-\cos(2\pi\alpha) - \sqrt 3 \sin(2\pi\alpha) - 2\sin(2\pi\beta)
\\ & = -2\pi[\sqrt 3 \alpha + 2 \beta] + 2\pi^2 \alpha^2 + O(\alpha^3) + O(\beta^3).
\end{align*}
The natural parameterisation here is $\alpha = a/3n$, $\beta = b/4n$ with $a \equiv -n \mod 3$ and $b \equiv -n \mod 4$.
Two examples of this type are shown in Figure~\ref{fig:11mod12}.

\begin{figure}
\input{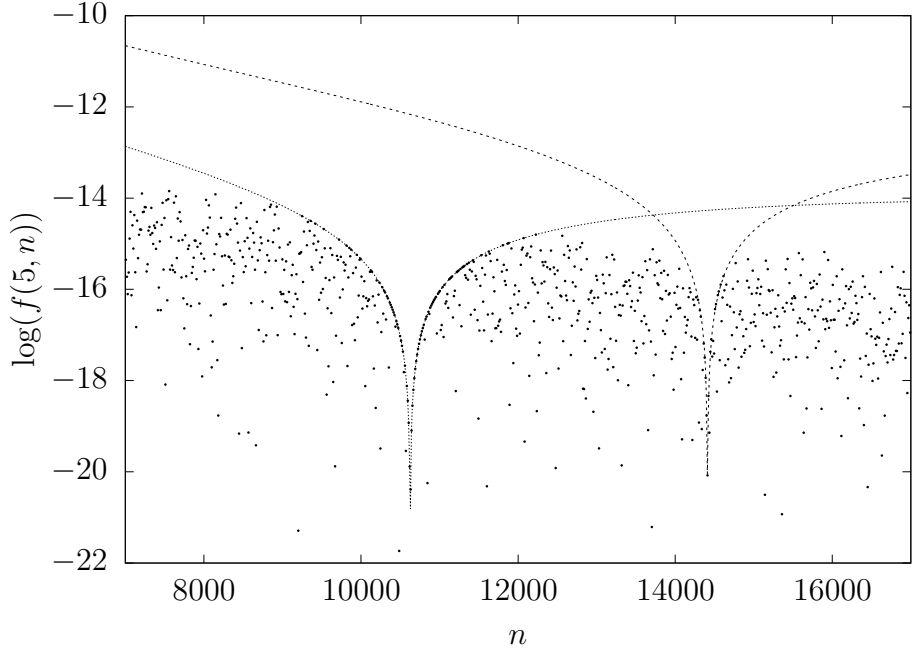}
\caption{$\log(f(5,n))$ for $7000 \leq n \leq 17000$, $n \equiv 11 \mod 12$ with upper bounds $\log|z_{3,\pm i}(13/3n,-15/4n)|$ and $\log|z_{3,\pm i}(-32/3n,37/4n)|$}
\label{fig:11mod12}
\end{figure}

The second configuration is rather different:
\begin{align}
z_{3,\pm 1}(\alpha,\beta) & = 1 + e(1/6 + \alpha) + e(-1/6-\alpha) + e(1/2 + \beta) + e(1/2-\beta)
\nonumber \\ & = 1 + 2c(1/6 + \alpha) + 2c(1/2 + \beta)
\nonumber \\ & = 1 + 2[c(1/6)c(\alpha) - s(1/6)s(\alpha)] - 2c(\beta)
\nonumber \\ & = 1 + \cos(2\pi\alpha) - \sqrt 3\sin(2\pi\alpha) - 2\cos(2\pi\beta)
\nonumber \\ & = -\frac{2\pi\sqrt 3a}{n} + \frac{2\pi^2(2b^2 - a^2)}{n^2}+ O(a^3/n^3) + O(b^4/n^4). \label{z3r}
\end{align}
When $a=0$ this degenerates to the configuration in Figure~\ref{fig:0mod6} witnessing the $O(1/n^2)$ bound for $n$ divisible by $6$.
Since $b$ does not appear in the linear term, there is no other way to make it small.
Instead we have to balance $an$ against $b^2$, so we become interested in the quadratic variant of Dirichlet's theorem.

The best bounds are due to Zaharescu.

\begin{theorem}[\cite{zaharescu}]\label{quadratic}
For every $Q \in \mathbb N$ and $\xi \in \mathbb R$, there are integers $p, q$ with $1 \leq q \leq Q$ such that
\[
\left|q^2\xi - p\right| \leq Q^{-4/7 + o(1)}.
\]
Moreover, there are infinitely many values of $p, q$ such that
\[
\left|q^2\xi - p\right| \leq q^{-2/3 + o(1)}.
\]
\end{theorem}

\begin{proposition}
For every $n$ divisible by $6$ there are integers $a, b$ such that
\[
0 < |z_{3,\pm 1}(a/n,b/n)| < n^{-11/8 + o(1)}.
\]
\end{proposition}

This is a worse bound than the $O(1/n^2)$ from Figure~\ref{thm:beat2}, but we present the analysis to illustrate what would be required to improve it.
For example, it improves on Theorem~\ref{thm:four-thirds} if it can be made compatible with the necessary congruence conditions when $n$ is not divisible by $6$.

\begin{proof}
Let $n$ be divisible by $6$ and write $\alpha = a/n$, $\beta = b/n$.
From \eqref{z3r},
\begin{align*}
z_{3,\pm 1}(\alpha,\beta) & = \frac{2\pi(2\pi b^2 - \sqrt 3an - \pi a^2)}{n^2} + O(a^3/n^3) + O(b^4/n^4).
\end{align*}

Let $1 \leq Q \leq n$ be an integer to be specified later.
By the first part of Theorem~\ref{quadratic} with $\xi = 2\pi/\sqrt 3 n$, there are $p, q$ with $1 \leq q \leq Q$ such that
\[
\left|2\pi q^2 - \sqrt 3 n p\right| \leq \sqrt 3 n Q^{-4/7 + o(1)}.
\]
Let $a = -p$, $b = q$.
Then 
\begin{align}
|z_{3,\pm 1}(\alpha,\beta)| & \leq 2\pi\sqrt 3/nQ^{4/7 + o(1)} + O(a^2/n^2)  + O(b^4/n^4) \nonumber
\\  & = 2\pi\sqrt 3/nQ^{4/7 + o(1)} + O(Q^4/n^4). \label{nQ47}
\end{align}
The optimal choice $Q = n^{21/32 + o(1)}$ gives $|z_{3,\pm 1}(\alpha,\beta)| = n^{-11/8 + o(1)}$.
\end{proof}

Unlike Dirichlet's theorem, which is tight in general (and for $\phi$ in particular), it is conjectured that the correct bound in Theorem~\ref{quadratic} is $Q^{-1+o(1)}$.
In that case the optimal choice of $Q$ in \eqref{nQ47} is $Q = n^{3/5 + o(1)}$, giving $|z_{3,\pm 1}(\alpha,\beta)| = n^{-8/5 + o(1)}$.

If we apply the second part of Theorem~\ref{quadratic} with $\xi = 2\pi/\sqrt 3$ we obtain infinitely many $p, q$ with 
\[
\left|2\pi q^2 - \sqrt 3 p\right| \leq \sqrt 3 q^{-2/3 + o(1)}.
\]
We can then take $a=-1$, $b=q$ and $n = p$ to obtain
\begin{align*}
|z_{3,\pm 1}(\alpha,\beta)| & \leq \sqrt 3/n^2q^{2/3 + o(1)} + O(1/n^2) + O(1/n^3) + O(q^4/n^4).
\end{align*}
Since $q = \Theta(\sqrt n)$, this is only an $O(n^{-2+o(1)})$ bound (infinitely often) since the error terms are dominating as in \eqref{eqn:swamp}, but it is by a different method from either Figure~\ref{fig:0mod6} or Theorem~\ref{thm:beat2} that may be more susceptible to improvement.

Finally, we mention that if we allow ourselves to go beyond $n$th roots of unity, there is a continuous family of real solutions to 
\[
1 + e(\alpha) + e(-\alpha) + e(\beta) + e(-\beta) = 1 + 2\cos(\alpha) + 2\cos(\beta) = 0.
\]
To take just one example, there is a $\theta$ such that $1 + 4\cos \theta = 0$.
Then there are integers $a, n$ such that $|\theta/2\pi - a/n| < 1/n^2$, whence $|1 + 4\cos(2\pi a/n)| = O(1/n^2)$ infinitely often.

\section{Computation}\label{sec:computation}

We end with some comments on computing $f(k,n)$.
The naive method is to fix one of the roots at $1$, then exhaust over all $\Theta(n^{k-1})$ choices for the other roots.
We can save two more powers of $n$ as follows.
Suppose that we have already chosen $k-2$ roots, with sum $y$.
Then the sum of roots $y + u + v$ is smallest when $u+v$ is closest to $-y$, a problem that can be solved in constant time using \eqref{eqn:sum-of-2} to solve for the arguments of $u$ and $v$ and trying a few different roundings.
This gives a final cost of $\Theta(n^{k-3})$ arithmetic operations to evaluate $f(k,n)$.
For $k=5$ this is a reasonable quadratic algorithm.
With a little more care we can cut down the constant factor.

\begin{proposition}
For $k=5$, it suffices to exhaust over optimal completions of the set
\[
\{1+e(a/n)+e(b/n) : 0 \leq 2a \leq b \leq 2n/5\}
\]
of approximately $n^2/25$ points.
\end{proposition}

\begin{proof}
By considering random rotations, it suffices to look at triples of points chosen from any arc of slightly more than $2/5$ of the circle.
By reflecting if necessary we may assume that the gap between the second and third points is at least as large as the gap between the first and second points;
finally, by rotating clockwise as far as possible we may assume that each triple includes $1$.
\end{proof}

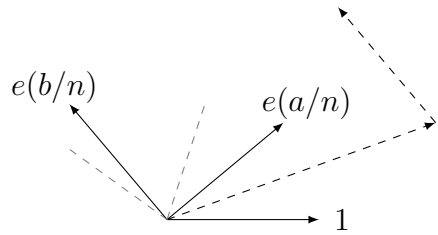
\begin{figure}
\begin{center}

\begin{tikzpicture}[scale=2]

\foreach \t in {0,130,40} 
  \draw[->] (0,0) -- (\t:1);
  
\draw[->,dashed] (0,0) -- ({1+cos(40)},{sin(40)});
\draw[->,dashed] ({1+cos(40)},{sin(40)}) -- ({1+cos(40)+cos(130)},{sin(40)+sin(130)});

\foreach \t in {144,72}   
  \draw[dashed,gray] (0,0) -- (\t:0.8);

\draw (0:1.15) node {$1$};
\draw (130:1.15) node {$e(b/n)$};
\draw (40:1.2) node {$e(a/n)$};

\end{tikzpicture}

\caption{$|1+e(a/n)+e(b/n)|$ is decreasing in $b$, since $e(b/n)$ always lies above $1+e(a/n)$.}
\label{fig:decreasing-b}
\end{center}
\end{figure}

There is a further saving available by observing that a significant fraction of these triple-sums will have length greater than $2$ plus the shortest length seen so far, so can't possibly have good completions.
Since $|1+e(a/n)+e(b/n)|$ is monotonically decreasing in $b$ (Figure~\ref{fig:decreasing-b}), we can easily exclude these triples for a further roughly one-third saving in running time.
Computing data for $n \leq 221000$ involved evaluating around $2^{47}$ sums over a period of several weeks.

For large $k$ it would be more efficient to generate the set $S$ of negatives of sums of $\lceil k/2 \rceil$ points (up to rotation) and check the sums of $\lfloor k/2 \rfloor$ points against sufficiently close elements of $S$.
This takes time at least $c_kn^{\lfloor k/2\rfloor}$, plus time and memory overhead for dealing with $S$.


\bibliography{roots}{}
\bibliographystyle{alpha}

\end{document}